\documentclass[12pt]{amsart}
\usepackage{amssymb, amscd, amsmath, amsthm, 
latexsym, enumerate}

\newtheorem{theorem}{Theorem}
\newtheorem{lemma}[theorem]{Lemma}
\newtheorem{cor}[theorem]{Corollary}

\newtheorem*{add}{Addendum}
\newtheorem*{defn}{Definition}

\begin{document}

\title{ $PD_3$-complexes bound}

\author{Jonathan A. Hillman }
\address{School of Mathematics and Statistics\\
     University of Sydney, NSW 2006\\
      Australia }

\email{jonathanhillman47@gmail.com}

\begin{abstract}
We show that every $PD_3$-complex $P$ bounds a $PD_4$-pair $(Z,P)$.
If $P$ is orientable we may assume that $\pi_1(Z)=1$.
We show also that if the inclusion of $Z$ into $P$ induces an isomorphism on fundamental groups then $\pi_1(Z)$ is a free group,
and that every $PD_3$-complex with a manifold 1-skeleton is homotopy equivalent to a closed 3-manifold.
\end{abstract}

\keywords{aspherical, boundary,  $PD_3$-complex, $PD_3$-group, $\pi_1$-injective}

\subjclass{57M45, 57P10}

\maketitle

It is well known that  every closed connected 3-manifold bounds  
a compact smooth 4-manifold (which may be assumed orientable
if the 3-manifold is orientable).
This follows from the calculation of the bordism rings, 
but there are also {\it ad hoc} low-dimensional proofs \cite{MK}.
There is an analogous notion of $PD$-bordism (as studied in \cite{HV}).
Much of the published work on this topic (and related notions,
such as $PD$-surgery and transversality) 
was driven by the needs and results of high-dimensional manifold topology,
and we have not found an explicit treatment of the low-dimensional cases.

In the very lowest dimensions $n=1$ or 2 every $PD_n$-complex  
$X$ is homotopy equivalent to a closed $n$-manifold, 
and $X$ bounds if and only if the corresponding manifold bounds.
Our interest is in the case $n=3$.
In \S1 we show that every $PD_3$-complex $P$ is the range of a homology
equivalence $f:M\to{P}$ with domain a closed 3-manifold.
The union of the mapping cylinder of this map with a suitable 4-manifold bounded by $M$ is the ambient space of a $PD_4$-pair with boundary $P$.
Some argument is needed, since there are $PD_3$-complexes 
which are not homotopy equivalent to closed 3-manifolds.
We use special features of the low-dimensional case,
and leave aside the general problem of Poincar\'e duality bordism.

The next two sections consider bordisms with fundamental group constraints.
In \S2 we adapt an argument of R. J. Daverman to show that if 
a $PD_4$-pair $(X,\partial{X})$ has connected, non-empty boundary and the 
inclusion of $\partial{X}$ induces a fundamental group isomorphism then 
$\pi_1(\partial{X})$ is a free group.
Every aspherical 3-manifold is the $\pi_1$-injective boundary 
of an aspherical 4-manifold \cite{DJW}, 
and in \S3 we introduce ``injective bordism" of $PD_n$-groups,
to put the corresponding question for $PD_3$-groups in a wider context.

In the final section we consider another aspect of the structure of $PD_3$-complexes:
we show that if a $PD_3$-complex has a manifold 1-skeleton then 
it is homotopy equivalent to a closed 3-manifold.
In all other dimensions it was known that every $PD_n$-complex has a
manifold 1-skeleton.

I would like to thank I. Hambleton for pointing out a looseness 
in my handling of framings in Theorem 1, 
and M.Land for his argument for the 4-dimensional case.

\section{$PD_3$-complexes}

If $f:M\to{P}$ is a degree-1 map from a 3-manifold $M$ to a 
$PD_3$-complex $P$ then surgery may be used to improve it 
to a $\mathbb{Z}[\pi]$-homology equivalence,
provided an obstruction in $L_3(\mathbb{Z}[\pi])$ vanishes \cite{KT}.
Here we need only a $\mathbb{Z}$-homology equivalence,
and the issue of promoting a degree-1 map to a normal map does not arise.
(In fact the orientation characters determine the normal fibrations in 
this dimension.)
The arguments of \cite{CS} probably apply to this situation,
since $L_3(\mathbb{Z})=0$, but we shall use naive, 
unobstructed surgery below the middle dimension, 
as in Theorem 5.1 of \cite{KM}.

\begin{theorem}
Orientable $PD_3$-complexes bound orientably.
\end{theorem}

\begin{proof}
Let $P$ be a $PD_3^+$-complex.
The fundamental class $[P]$ may be represented by a 3-cycle 
$\Sigma\psi_i$, where each summand $\psi_i$ is a singular 3-simplex.
Since $\partial\Sigma\psi_i=0$,
the faces of the summands must match in pairs.
Choosing such a pairing gives a map $f:E\to{P}$ with domain 
a finite 3-complex which is an orientable  3-manifold away from its vertices. 
The finitely many non-manifold points have 
links which are aspherical orientable surfaces.
After replacing conical neighbourhoods of such vertices by handlebodies,
we obtain a degree-1 map $g:M\to{P}$ with domain 
a closed orientable 3-manifold. 
(The existence of such a map also follows from the Atiyah-Hirzebruch 
spectral sequence for oriented bordism.)

Since $g$ is a degree-1 map,
$\pi_1(g)$ is an epimorphism, 
and so it maps $\pi_1(M)'$ onto $\pi_1(P)'$.
Hence the Hurewicz homomorphism maps $\mathrm{Ker}(\pi_1(g))$
onto the ``homology surgery kernel" $K_1=\mathrm{Ker}(H_1(g;\mathbb{Z}))$. 
Moreover, $K_1$ is a direct summand of $H_1(M;\mathbb{Z})$.
Let $L:S^1\to{M}$ be an embedding which represents a generator of 
a direct summand of $K_1$, 
and let $N\cong{S^1}\times{D^2}$ be a regular neighbourhood of $L$.
Then $H_1(\partial{N};\mathbb{Z})$ has a longitude-meridian basis 
$\{\lambda,\mu\}$,
where $\lambda$ is freely homotopic in $M$ to $L$
and $\mu$ bounds a transverse disc.
(This property characterizes the meridian, up to sign.
There is no canonical choice of longitude,
and we shall make use of some freedom of choice below.)
Since $g\circ{L}$ is null-homotopic, so are $g|_N$ and $g|_{\partial{N}}$.
Therefore $g|_{\overline{M\setminus{N}}}$ extends to a degree-1 map
$g':M'\to{P}$,  for any 3-manifold $M'$ obtained by Dehn surgery on $L$.

If the homology class of $L$ has infinite order then there is 
a closed orientable surface $S$ in $M$ which meets 
the image of $L$ transversely in one point, by Poincar\'e duality.
We may assume that  $\partial(S\cap{N})=\mu$.
Hence $\mu$ is null-homologous in $\overline{M\setminus{N}}$,
since it bounds $\overline{S\setminus{N}}$.
Let $\phi:\partial{D^2}\times{S^1}\to\partial\overline{M\setminus{N}}$
be a homeomorphism which maps $\partial{D^2}\times{1}$ to $\lambda$,
and let $M'=\overline{M\setminus{N}}\cup_\phi{D^2}\times{S^1}$.
Since $g|_{\partial{N}}$ is null-homotopic,
$g|_{\overline{M\setminus{N}}}$ extends to a degree-1 map
$g':M'\to{P}$, and $\lambda$ and $\mu$ are each null-homologous in $M'$.
Hence $H_1(M';\mathbb{Z})\cong{H_1(M;\mathbb{Z})}/\langle\lambda\rangle$. 
Proceeding in this way, we may arrange that $K_1$ is finite.

If $L$ represents a finite direct summand of $K_1$ then
the image of the meridian $\mu$  in 
$H_1(\overline{M\setminus{N}};\mathbb{Z})$ has infinite order.
(For if the image of $\mu$ has finite order $s$ in 
$H_1(\overline{M\setminus{N}};\mathbb{Z})$ then $s\mu$ 
would bound a surface $F\in\overline{M\setminus{N}}$, 
and so $L$ would have non-zero intersection number $\pm{s}$ 
with the closed surface $\widehat{F}=F\cup{s}D^2$.
Hence the image of $L$ has infinite order in $H_1(M;\mathbb{Z})$.)
It is an easy consequence of Poincar\'e duality that 
the image of $H_1(\partial{N};\mathbb{Q})\cong\mathbb{Q}^2$ in
$H_1(\overline{M\setminus{N}};\mathbb{Q})$ has rank 1.
Hence we may choose the longitude $\lambda$ so that its image
in $H_1(\overline{M\setminus{N}};\mathbb{Z})$ has finite order.
Let $M''=M\cup_f{S^1\times{D^2}}$, 
where $f$ maps the first factor of $S^1\times{S^1}$
to $\mu$ and the second to $\lambda$.
Then $g$ extends to a degree-1 map from $M''$ to $P$
and $H_1(M'';\mathbb{Z})/\mathrm{Im}(\langle\mu\rangle)\cong
{H_1(M;\mathbb{Z})/\mathrm{Im}(\langle\lambda\rangle)}$.
This reduces the torsion subgroup of the homology surgery kernel,
at the cost of increasing the rank.
We then apply the earlier argument, 
to reduce the rank without further changing the torsion subgroup.
After several iterations of these steps,
we reduce the homology surgery kernel to 0.

Thus we may assume that  $H_1(g;\mathbb{Z})$ is an isomorphism.
Hence $g$ induces isomorphisms  on homology in all degrees.
Let $W$ be a compact orientable 4-manifold with boundary $M$.
After elementary surgeries on a basis for $\pi_1(W)$,
we may assume also that $\pi_1(W)=1$.
Let $MCyl(g)$ be the mapping cylinder of $g$ and let $Z=W\cup_MMCyl(g)$.
Then $\pi_1(Z)=1$, since $\pi_1(g)$ is an epimorphism,
and the inclusions $j:W\to{Z}$ and $J:(W,M)\to(Z,MCyl(g))\simeq(Z,P)$
induce isomorphisms on homology.
Since $J_*(j^*\xi\cap[W,M])=\xi\cap{J_*[W,M]}$
for all $\xi\in{H^i(Z;\mathbb{Z})}$ and $i\geq0$
and since $(W,M)$ is an orientable 4-manifold pair,
it follows that $(Z,P)$ is an orientable $PD_4$-pair with boundary $P$.
\end{proof}

Arranging that $g$ be an integral homology equivalence seems
to be a necessary first step.
If a pair $(X,Y)$ satisfies Poincar\'e-Lefshetz duality 
with simple coefficients $R$ and one component $Y_1$ of $Y$ is 
a deformation retract of $X$ then it follows from the long exact sequence 
of homology for the pair and duality that $Y$ must have two components,
$Y=Y_1\sqcup{Y_2}$ say, 
and that $H_i(Y_2;R)\cong{H_i(Y_1;R)}\cong{H_1(X;R)}$ for all $i\geq0$.
On the other hand,  the inclusion of $Y_2$ into $X$ need not induce an isomorphism on fundamental groups.
Every homology 3-sphere $\Sigma$ bounds a contractible 4-manifold, 
$C$ say \cite[Corollary 9.3C]{FQ},
and removing a 4-ball from the interior of $C$ leaves
an homology cobordism from $\Sigma$ to $S^3$.

In general,  $(MCyl(g), M\sqcup{P})$ need not be a $PD_4$-pair,
even if $g$ is an integral homology equivalence.
Let $M$ be the flat 3-manifold with holonomy of order 6.
Then there is an integral homology equivalence $g:M\to {P}=S^2\times{S^1}$.
The mapping cylinder $MCyl(g)$ fibres over $S^1$,
with fibre $MCyl(g')$ the mapping cylinder of the degree-1 collapse 
$g':T\to{S^2}$.
If $(MCyl(g), M\sqcup{P})$ were a $PD_4$-pair then 
$(MCyl(g'), T\sqcup{S^2})$ would be a 1-connected $PD_3$-pair.
But this is clearly not the case.

We make no use of the fact that a $PD_3$-complex is finitely dominated,
and much of the above argument extends to the non-orientable case.
Here we must first find a 
$\mathbb{Z}[\mathbb{Z}^\times]$-homology equivalence.
(We note also that $L_3(\mathbb{Z}[\mathbb{Z}^\times],w)=0$
if $w$ is nontrivial,
but we do not use this fact.)

\begin{add}
Every non-orientable $PD_3$-complex bounds.
\end{add}

\begin{proof}
Let $P$ be a $PD_3$-complex such that $w=w_1(P)\not=0$.
We could represent a generator of $H_3(P;\mathbb{Z}^w)$ 
by a  geometric cycle $f:E\to{P}$ with singularities only at the vertices.
However in this case it is not obvious that the links of the vertices 
have even Euler characteristic.
Instead we appeal to the Atiyah-Hirzebruch spectral sequence
for $w$-twisted bordism.
This gives a 3-manifold $M$ and a map ${g:M\to{P}}$ such that 
$g^*w=w_1(M)$ and  $H_3(g;\mathbb{Z}^w)$ is an isomorphism.
The surgery kernel $\mathrm{Ker}(\pi_1(g))$ is
represented  by knots with product neighbourhoods.
Let $M^+$ be the orientable 2-fold covering space of $M$.
After elementary surgeries as in the theorem, 
we may assume that the image of $\mathrm{Ker}(\pi_1(g))$ in
$H_1(M^+;\mathbb{Z})$ is trivial, 
and so $g$ is a $\mathbb{Z}[\mathbb{Z}^\times]$-homology equivalence.
We may also assume that $M=\partial{W}$ where 
$w_1(W):\pi_1(W)\to\mathbb{Z}^\times$ is an isomorphism.
The rest of the argument is then as in the theorem.
\end{proof}

It is unlikely that such arguments extend much further.
There are non-orientable $PD_n$-complexes and orientable $PD_{n+1}$-complexes whose Spivak normal bundle has no TOP reduction, for all $n\geq4$ \cite{Ha,La17}.
Such complexes admit no degree-1 normal maps with domain a
closed $n$-manifold. 

However, M. Land has suggested the following argument for
the case $n=4$.
If $P$ is a $PD_n$-complex for some $n\geq4$ then $P\simeq{H}\cup{Y}$,
where $H$ is a 1-handlebody,
$(Y,\partial{H})$ is a $PD_n$-pair and the inclusion induces 
an epimorphism from $\pi_1(H)$ onto $\pi_1(P)$ \cite[Lemma 2.8]{Wa}.
We may perform elementary surgeries on a basis for $\pi_1(H)$ inside
the manifold 1-skeleton $H$, 
and the trace $W$ of the surgeries is a $PD$-bordism from $P$ 
to a 1-connected $PD_n$-complex.
If $P$ is orientable then $W$ is orientable.
Every 1-connected $PD_4$-complex is homotopy equivalent to a
TOP 4-manifold \cite[\S11.4]{FQ}.
Now $\Omega_4^{STOP}\cong\mathbb{Z}\oplus\mathbb{Z}/2\mathbb{Z}$,
where the first summand is detected by the signature $\sigma$ and
the second by the Kirby-Siebenmann invariant $KS$.
Let $*CP^2$ be the fake projective plane,
with $KS(*CP^2)\not=0$,
and let $X=\overline{CP^2}\#*CP^2$.
Then $\sigma(X)=0$ and $KS(X)\not=0$.
The signature is an invariant of PD bordism, 
but $X$ bounds as a $PD_4$-complex, since $*CP^2\simeq{CP^2}$.
Hence the signature defines an isomorphism 
$\Omega_4^{SPD}\cong\mathbb{Z}$.
In the non-orientable case we find that 
$\Omega_4^{PD}\cong(\mathbb{Z}/2\mathbb{Z})^2$,
detected by SW numbers.

\section{fundamental group isomorphisms}

R. J.  Daverman showed that if a compact 4-manifold $M$ has connected,
non-empty boundary  and the inclusion of the boundary induces an 
isomorphism of fundamental groups then $\pi_1(\partial{M})$ is a free group. 
We shall show that a similar statement holds for $PD_4$-pairs.
The key observation of \cite{Da94} is that since $Y$ is a boundary, 
the image of $H_3(Y;F)$ in $H_3(X;F)$ is 0, for all simple coefficients $F$.
The argument is then essentially homological.
The only novelty in our situation is that there are indecomposable $PD_3$-complexes with infinite, virtually free fundamental groups
\cite[Chapters 6 and 7]{Hi}.

\begin{lemma}
\label{Davor}
Let $P$ be an indecomposable orientable $PD_3$-complex such that 
$\pi=\pi_1(P)$ is virtually free but not free.
Then $\pi$ has a subgroup of finite index which has a nontrivial finite 
cyclic group as a free factor.
\end{lemma}

\begin{proof}
We have $\pi=\pi\mathcal{G}$, where $(\mathcal{G},\Gamma)$ is  a 
 finite graph of finite groups,
in which the underlying graph $\Gamma$ is a tree, 
with $n$ vertices say,
and at most one of the edge groups is not $\mathbb{Z}/2\mathbb{Z}$
\cite[Theorem 6.10]{Hi}.
If $\pi$ is finite then the claim is clear, so we
we may assume that $n\geq2$.
We may also assume that the first $n-2$ vertex groups are dihedral groups
$G_i\cong{D_{2m_i}}$, where $m_i\geq3$ and is odd,
while $G_{n-1}\cong{D_{2m}}$ or $D_{2m}\times\mathbb{Z}/3\mathbb{Z}$,
and $G_n$ has cohomological period dividing 4.
Let $K$ be the normal closure of the union of the commutator subgroups 
$\cup_{i<n}G_i'$. Then $\pi/K\cong{G_n}$, and so $K$ has finite index in $\pi$.
Each $G_i'$ with $i<n$ is cyclic, and is a free factor of $K$.
\end{proof}
 
\begin{lemma}
\label{Davfree}
Let $P$ be an indecomposable $PD_3$-complex such that $\pi=\pi_1(P)$ 
has a free subgroup of index $\leq2$.
Then $P$ is homotopy equivalent to one of $S^3$, $RP^3$, 
$RP^2\times{S^1}$, $S^2\times{S^1}$ or $S^2\tilde\times{S^1}$. 
\end{lemma}

\begin{proof}
We may assume that $\pi=\pi\mathcal{G}$,
where $(\mathcal{G},\Gamma)$ is a finite graph of finite groups.
If there is a nontrivial vertex group it must have order 2,
and then $\pi=\mathbb{Z}/2\mathbb{Z}$
or $\mathbb{Z}/2\mathbb{Z}\oplus\mathbb{Z}$, 
since all the edge groups have order 2 and $\pi$ is not a proper free product.
If $\pi=\mathbb{Z}/2\mathbb{Z}$ then $P\simeq{RP^3}$.
If $\pi\cong\mathbb{Z}/2\mathbb{Z}\oplus\mathbb{Z}$
then $P\simeq{RP^2}\times{S^1}$.
If all vertex groups are trivial then $\pi\cong\mathbb{Z}$ or 1, 
and so $P\simeq{S^2}\times{S^1}$, $S^2\tilde\times{S^1}$
or $S^3$.
(See \cite[Theorems 5.2 and 5.8]{Hi}.)
\end{proof}

\begin{theorem}
Let $(X,Y)$ be a $PD_4$-pair with connected, non-empty boundary $Y$
and let $j$ be the inclusion of the $Y$ into $X$.
If $\pi_1(j)$ is an isomorphism then $\pi=\pi_1(Y)$ is a free group.
\end{theorem}

\begin{proof}
Assume first that $X$ is orientable.
Then $Y$ is also orientable, 
and so it either has an aspherical summand or
$\pi$ is virtually free.
If $Y$ has an aspherical summand the argument of \cite[Theorem A]{Da94}
applies.
If $\pi$ is virtually free then we may pass to the finite covering space
associated to a subgroup as in Lemma \ref{Davor}.
In either case, we obtain a contradiction,
and so $\pi$ must be free.

If $X$ is not orientable then the orientable double cover $(X^+,Y^+)$
satisfies a similar hypothesis, and so $\pi^+=\pi_1(Y^+)$ must be a free group.
It then follows from Lemma \ref{Davfree} and 
\cite[Theorems 7.1 and 7.4]{Hi} that $Y$ must be
a connected sum of copies of $RP^3$, $RP^2\times{S^1}$, $S^2\times{S^1}$ and $S^2\tilde\times{S^1}$.  
We may use the earlier argument with coefficients $F=\mathbb{F}_2$.
Then $X$ and $Y$ are $\mathbb{F}_2$-orientable, 
and $H_3(c_Y;\mathbb{F}_2)=0$, and so we may exclude $RP^3$ and $RP^2\times{S^1}$ summands, as before.
Hence $\pi$ must again be free.
\end{proof}

As observed by Daverman, the argument extends readily to the case when 
$\pi_1(j)$ is a monomorphism with image a retract of $\pi_1(X)$.

We shall give an analogous result for $PD_3$-pairs in \S4 below.

\section{asphericity and $\pi_1$-injectivity}

Relative hyperbolization may be used to show that every 
closed orientable triangulable $n$-manifold is orientable cobordant to
an aspherical $n$-manifold \cite{DJ91}, 
and that every aspherical $n$-manifold which is the boundary 
of a triangulable $(n+1)$-manifold is in fact the $\pi_1$-injective boundary 
of an aspherical $(n+1)$-manifold. 
Similarly, every pair of aspherical $n$-manifolds which are cobordant
together bound an aspherical $(n+1)$-manifold $\pi_1$-injectively
\cite{DJW}; see also \cite[Theorem 5.1]{Ha81}.

In the lowest dimensions $n=1$ or 2 we may avoid hyperbolization
(at least for the orientable cases).
If $n=1$ then $S^1$ is the boundary of the once-punctured torus $T_o$,
and the inclusion of $S^1$ into $T_o$ is $\pi_1$-injective.

If $n=2$  then $T=\partial{T_o\times{S^1}}$,
and the inclusion of $T$ into $T_o\times{S^1}$ is $\pi_1$-injective.
The Klein bottle bounds the mapping torus of 
an orientation-reversing involution of $T_o$. 
The exterior of the $\Theta$-graph in $S^3$ depicted in \cite[Figure 3.10]{Th} 
has a hyperbolic structure for which the boundary is totally geodesic
(and hence incompressible) \cite[Example 3.3.12]{Th}.
The boundary has genus 2, 
and suitable finite cyclic covers are orientable hyperbolic 3-manifolds with 
connected totally geodesic boundary of arbitrary genus $g>1$.

Since every 3-manifold bounds,  every aspherical 3-manifold is the
$\pi_1$-injective boundary of an aspherical 4-manifold,
by the result of \cite{DJW}.
On the other hand,  for most values of $n\geq4$ there are aspherical 
$n$-manifolds which do not bound at all
(since $\Omega_n^{SO}\not=0$ if $n\geq8$).
In particular,  there are $\mathbb{H}^2(\mathbb{C})$-manifolds $M$ 
with the rational cohomology of $\mathbb{CP}^2$ \cite{PY}.
No such $M$ can bound (even as an unoriented $PD_4$-complex),
since $\chi(M)=3$ is odd.
Iterated products of such manifolds give counterexamples 
in all dimensions for which $\Omega_n^{SO}$ is infinite.

\begin{defn}
A $PD_n$-group $G$  bounds injectively if it is a subgroup of a group $\pi$ 
such that  $(\pi,G)$ is a  $PD_{n+1}$-pair of groups. 

Two $PD_n$-groups $G_1$ and $G_2$ are  injectively bordant
if they are subgroups of a group $\pi$ such that $(\pi,G_1,G_2)$ is 
a $PD_{n+1}$-pair of groups.
\end{defn}

Injective bordism of $PD_n$-groups is an equivalence relation.
Reflexivity is displayed by the pair $(G,G,G)$ with $\pi=G_1=G_2=G$, 
symmetry is obvious and transitivity follows from  \cite[Theorem 8.1]{BE}.

It remains an open question whether every finitely presentable $PD_n$-group
is the fundamental group of an aspherical closed $n$-manifold.
However,  we allow the possibility that $G$ or the ambient group $\pi$
of a pair $(\pi,G)$ is of type $FP$,  but not finitely presentable.
For every $n\geq4$ there are uncountable many $PD_n$-groups 
which are $FP$ but not finitely presentable \cite{Le}.
(The case $n=3$ remains open.)

These observations suggest the following questions.
\begin{enumerate}
\item{}Does every $PD_3$-group $G$ bound injectively?
\item{}Is every $PD_n$-group $G$ injectively bordant to a finitely presentable $PD_n$-group?
\end{enumerate}

If every $PD_3$-group is a 3-manifold group
then (1) has a positive answer.
The second question is purely speculative.

\section{manifold 1-skeleton}

A {\it cube with handles} is a 3-manifold
obtained by attaching 1-handles to $D^3$.
A connected $PD_3$-complex $X$ has a {\it manifold 1-skeleton\/} if 
$X\simeq{H}\cup_\phi{Y}$, 
where $H$ is a {\it cube with handles} and the inclusion of $H$ into $X$ 
is 1-connected (i.e., it induces an epimorphism from $\pi_1(H)$ to $\pi_1(X)$),
and $\phi:\partial{H}\to{B}$ is a homotopy equivalence.
It is clear on homological grounds that $Y$ and $B$ must be non-empty,
and $(Y,B)$ is then a $PD_3$-pair.
(See \cite[Proposition 2.7]{Wa}, ignoring the role of simplicity there.)

These notions extend easily to all dimensions,
and if $n\geq4$ then every $PD_n$-complex 
has a manifold 1-skeleton \cite[Lemma 2.8]{Wa}.
(See also \cite[Proposition 2.3]{Ho}.)
This is also the case if $n\leq2$, 
for then every $PD_n$-complex is 
homotopy equivalent to a closed $n$-manifold.
We shall show that the case $n=3$ is exceptional.

\begin{theorem}
Let $(Y,B)$ be a $1$-connected $PD_3$-pair.
Then $Y$ is aspherical and $\pi=\pi_1(Y)$ is a free group.
\end{theorem}

\begin{proof}
Let $\Gamma=\mathbb{Z}[\pi]$.
Since $(Y,B)$ is 1-connected, 
$Y$ may be obtained (up to homotopy equivalence) 
by adding cells of dimension $>1$ to $B$.
Therefore $H^i(Y,B;\Gamma)=H_i(Y,B;\Gamma)=0$, for $i\leq1$.
Hence $H_j(Y;\Gamma)=0$  for $j>1$,  by Poincar\'e duality,
and so $Y$ is aspherical.
Similarly,  if $\mathcal{M}$ is any left $\mathbb{Z}[G]$-module
then $H^j(Y;\mathcal{M})=0$ for $j>1$,
and so $c.d.\pi\leq1$.
Thus $\pi$ is a free group.
\end{proof}

\begin{cor}
Let $X$ be a connected $PD_3$-complex with a manifold $1$-skeleton.
Then $X$ is homotopy equivalent to a closed $3$-manifold.
\end{cor}

\begin{proof}
By assumption, $X\simeq{H}\cup_\phi{Y}$, 
where $H$ is a cube with handles
and $(X,H)$ is a 1-connected pair.
Since $H$ is a cube with handles, 
$\partial{H}$ is a connected closed surface.
The inclusion of $B$ into $Y$ is also 1-connected, 
by excision and the Van Kampen Theorem.

Since $(Y,B)$ is a $PD_3$-pair,
$Y$ is aspherical and $\pi_1(Y)$ is a free group, 
by Theorem 5.
Hence there is a homotopy equivalence $f:(Y,B)\simeq(H',\partial{H'})$,
where $H'$ is a second cube with handles, 
by \cite[Theorem 3.12]{Hi} and its
extension to the non-orientable case \cite[page 38]{Hi}.
Since $\partial{H}$ and $\partial{H'}$ are closed surfaces, 
the homotopy equivalence $f\circ\phi:\partial{H}\to\partial{H'}$
is homotopic to a homeomorphism $\Phi$,
by the Dehn-Nielsen Theorem \cite[Theorem 5.6.1]{ZVC}, 
and so $X$ is homotopy equivalent to the 3-manifold $H\cup_\Phi{H'}$.
\end{proof}

There are finite $PD_3$-complexes which are not homotopy equivalent 
to 3-manifolds. 
The first and simplest such example has fundamental group $S_3$.,
and was constructed by R. G. Swan, in another context, 
before the notion of $PD_n$-complex had been defined.
(See \cite[Chapters 5 and 6]{Hi}, and the references there.)

\end{document}